\documentclass[12pt]{amsart}

\usepackage{amssymb}
\usepackage{mathrsfs}

\newtheorem{thm}{Theorem}[section]
\newtheorem{prop}[thm]{Proposition}
\newtheorem{problem}[thm]{Problem}
\theoremstyle{remark}
\newtheorem{remark}[thm]{Remark}

\newcommand\Afrak{\mathfrak{A}}
\newcommand\Bfrak{\mathfrak{B}}
\newcommand\Cfrak{\mathfrak{C}}
\newcommand\Gfrak{\mathfrak{G}}
\newcommand\Xfrak{\mathfrak{X}}

\newcommand\efrak{\mathfrak{e}}

\newcommand\id{\operatorname{id}}
\newcommand\iso{\simeq}

\newcommand\Bscr{\mathscr{B}}
\newcommand\Cscr{\mathscr{C}}
\newcommand\Gscr{\mathscr{G}}
\newcommand\Mscr{\mathscr{M}}
\newcommand\Xscr{\mathscr{X}}
\newcommand\Yscr{\mathscr{Y}}
\newcommand\Fbb{\mathbb{F}}
\newcommand\Rbb{\mathbb{R}}
\newcommand\Qbb{\mathbb{Q}}
\newcommand\Zbb{\mathbb{Z}}

\newcommand\binomial[2]{\genfrac{(}{)}{0pt}{}{#1}{#2}}
\newcommand\cbinomial[2]{\genfrac{\langle}{\rangle}{0pt}{}{#1}{#2}}
\newcommand\Seq[1]{\langle #1 \rangle}
\newcommand{\Aut}{\operatorname{Aut}}
\newcommand\Fol{\textrm{F\o l}}
\newcommand{\symdif}{\triangle}
\newcommand{\C}{\mathtt{c}}
\newcommand\Age{\operatorname{Age}}

\title
{Amenability and Ramsey Theory}



\thanks{
I would like to thank Louis Billera for a helpful discussion 
on balanced sets and for pointing me to the literature on the topic.
I would also like to thank Paul Larson, Lionel Nguyen van Th\'e, and Todor Tsankov for reading
drafts of this paper and offering comments.
The research presented in this paper was partially supported by
NSF grant DMS--0757507.
Any opinions, findings, and conclusions or recommendations expressed in this article
are those of the author
and do not necessarily reflect the views of the National Science Foundation. 
}

\author{Justin Tatch Moore}

\email{{\tt justin@math.cornell.edu}}

\address{Justin Moore \\
555 Malott Hall \\
Department of Mathematics \\
Cornell University \\
Ithaca, NY 14853-4201 }

\begin{document}

\begin{abstract}
The purpose of this article is to connect the notion of the amenability of
a discrete group with a new form of structural Ramsey theory.
The Ramsey theoretic reformulation of amenability constitutes
a considerable weakening of the F\o lner criterion.
As a by-product, it will be shown that in any non amenable group $G$,
there is a subset $E$ of $G$ such that no finitely additive probability measure
on $G$ measures all translates of $E$ equally.
\end{abstract}

\maketitle

\section{Introduction}

A group $G$ is \emph{amenable}\footnote{Throughout this paper,
the adjective ``left''
is implicit in the usage of \emph{action},
\emph{amenable}, \emph{F\o lner}, \emph{invariant},
and \emph{translation} unless otherwise stated.}
if there is a finitely additive, translation invariant probability measure
defined on all subsets of $G$.
This notion was isolated by von Neumann from the Banach-Tarski paradox.
Since then
it has played an important role in a diverse cross section of mathematics.
It has a large number of seemingly different equivalent formulations (see \cite{amenability:Paterson}, \cite{Banach-Tarski:Wagon});
two of the most celebrated are:
\begin{thm} \cite{alg_fass} \cite{card_alg}
A group $G$ is amenable if and only if there do
not exist elements $g_i$ $(i < k)$ of $G$ and  a partition of $G$ into sets
$A_i$ $(i < k)$ such that, for some $i_0 < i$,
both $\{g_i    A_i : i < i_0\}$ and $\{g_i    A_i : i_0 \leq i < k\}$ are partitions of $G$.
\end{thm}

\begin{thm} \cite{Folner_crit} (see also \cite{Folner:Namioka})
A group $G$
is amenable if and only if for every finite $A \subseteq G$ and every $\epsilon > 0$
there is a finite $B \subseteq G$ such that (letting $\symdif$ denote symmetric difference)
\[
\sum_{a \in A} |(a    B) \symdif B| \leq \epsilon |B|
\]
\end{thm}

The set $B$ which satisfies the conclusion of this theorem is said to be \emph{$\epsilon$-F\o lner
with respect to $A$};
the assertion that such sets exist for each $\epsilon > 0$ is known as the \emph{F\o lner criterion}.

One of the main results of the present article is to formulate a weaker criterion for amenability
than the F\o lner criterion.
If $A$ is a set, let $P(A)$ denote the collection of all finitely additive probability measures on $A$.
If $G$ is a group, then the operation on $G$ is extended to
$\ell^1 (G)$ bilinearly:
\[
\mu    \nu (A) = \sum_{xy \in A} \mu(\{x\}) \nu (\{y\})
\]
(Here and throughout we identify $G$ with both a subset of $\ell^1(G)$
and a subset of $P(G)$ by regarding its elements as point masses.)
Observe that $g   \nu (E) = \nu (g^{-1} E)$.

If $A$ and $B$ are finite subsets of $G$ and $\epsilon > 0$, then
$B$ is \emph{$\epsilon$-Ramsey with respect to $A$} if whenever
$E \subseteq B$, there is a $\nu$ in $P(B)$ such that
\begin{itemize}

\item $P(A) \nu \subseteq P(B)$ and

\item $|\mu    \nu (E) - \mu'    \nu (E)| \leq \epsilon$
for all $\mu$ and $\mu'$ in $P(A)$.

\end{itemize}
Notice that one obtains an equivalent statement if
$\mu$ ranges over the elements of $A$
--- these are the extreme points of $P(A)$.
Also observe that if $\nu$ is a finitely supported probability measure on $G$, then
$P(A)    \nu$ can be regarded as a copy of $P(A)$.
Thus $B$ is $\epsilon$-Ramsey with respect to $A$
if whenever we induce a linear coloring
of $P(B)$ by assigning the values $0$ and $1$ to the elements of $B$, there is
a copy of $P(A)$ on which the coloring is $\epsilon$-monochromatic. 

\begin{thm} \label{amen_criterion}
Let $G$ be a group.
The following are equivalent:
\begin{enumerate}

\item \label{inv_measure:intro}
For every $E \subseteq G$ and every finite $A \subseteq G$, there is a $\mu$ in
$P(G)$ such that $\mu(g    E) = \mu (E)$ for all $g$ in $A$.

\item \label{ramsey:intro}
For every finite $A \subseteq G$, there is a $B$ which is $\frac{1}{2}$-Ramsey
with respect to $A$.

\item \label{0-ramsey:intro}
For every finite $A \subseteq G$, there is a $B$ which is $0$-Ramsey
with respect to $A$.

\item \label{amen:intro}
$G$ is amenable.

\end{enumerate}
\end{thm}

The equivalence of (\ref{inv_measure:intro})
and (\ref{amen:intro}) was unexpected and while they are purely global statements
about $G$ involving its typically infinite subsets,
the proof crucially employs the finitary interpolation provided by
(\ref{ramsey:intro}).
Also notice that the only examples of
$0$-F\o lner sets are the trivial ones:
a finite group is a $0$-F\o lner set in itself.
Thus (\ref{0-ramsey:intro}) represents a new phenomenon
for which the F\o lner criterion
provides no analog.
The proof of Theorem \ref{amen_criterion} will also provide a quantitative relationship between
$\epsilon$-Ramsey sets and $\epsilon$-F\o lner sets;
this is the content of Section \ref{Ramsey_function:sec}.

Let us say that a subset $E$ of a group $G$ is \emph{invariantly measurable}
if there is a $\mu$ in $P(G)$ such that $\mu (g    E) = \mu (E)$ for all $g$ in $G$.
That the invariant measurability of sets can be witnessed by a single measure turns out
\emph{not} to be a phenomenon present in arbitrary groups.

\begin{thm} \label{inv_meas_c.e.}
There are five invariantly measurable subsets of $\Fbb_2$
which can not be simultaneously measured invariantly.
\end{thm}

Theorem \ref{amen_criterion}
tells us that if a group $G$ is not amenable,
then there is a single set $E \subseteq G$ which can not be measured invariantly.
It is natural to ask to what extent $E$ can be specified by a finite amount of information.
Let $A$ be a fixed finite subset of $G$.
If $E \subseteq G$, define
\[
X^A_{E} (g) = \{a \in A : ag \in E\}
\]
\[
\Xscr^A_E = \{X_E(g) : g \in G\}.
\]
(When $A$ is clear from the context, the superscript will be suppressed.)
Thus if $A$ is a ball about the identity, $X_E (g)$ is a ``picture'' of $E$ centered at $g$
where the scope of the image is specified by $A$.
The set $\Xscr_E$ is then the collection of all such pictures of $E$ taken from different vantage points in $G$.
If $\Yscr$ is a collection of subsets of $A$, then we say that $\Yscr$ is \emph{realized} in $G$ if
$\Yscr = \Xscr_E$ for some $E \subseteq G$.

A collection $\Yscr$ of subsets of a finite set $A$ is \emph{$\epsilon$-balanced} if there is a convex combination $v$
of  the characteristic functions of its elements such that $\max (v) - \min (v) \leq \epsilon$.
\emph{Balanced} will be taken to mean $0$-balanced.
It follows from the Hahn-Banach Separation Theorem that a collection $\Yscr$ is
unbalanced if and only if there is an $f:A \to \Rbb$ such that
$\sum_{a \in A} f(a) = 0$ and
for every $Y$ in $\Yscr$,
$\sum_{a \in Y} f(a) > 0$.

We will prove the following analog of Theorem 4.2 of \cite{tiling_curv_amen}.

\begin{thm} \label{puzzle}
For a group $G$, the following are equivalent:
\begin{enumerate}

\item \label{non_amen:intro}
$G$ is non amenable.

\item \label{unbalanced:intro}
There is a finite $A \subseteq G$
and an unbalanced collection $\Yscr$ of subsets of $A$ which is realized in $G$.

\item \label{no_0-ramsey:intro}
There is a finite $A \subseteq G$ for which there is no finite $B$
which is $0$-Ramsey with respect to $A$.

\end{enumerate}
\end{thm}
Thus in order to establish the non amenability of a group, it
is sufficient to realize a subcollection of
\[
\Yscr_f = \{Y \subseteq A : \sum_{a \in Y} f(a) > 0\}
\]
for some $f:A \to \Rbb$ such that $\sum_{a \in A} f(a) = 0$.
Balanced sets have been studied in game theory
(e.g. \cite{ind_method_min_balanced}),
although the focus has been on minimal balanced collections
rather than the maximal unbalanced
collections, which are most relevant to the present discussion.

The above theorems concern the amenability of discrete groups.
The amenability of a topological group $G$ can be formulated as follows:
whenever $G$ acts continuously on a compact space $K$,
$K$ supports an\ Borel probability measure which is preserved by the action of $G$.
A strengthening of amenability in this context is that of \emph{extreme amenability}:
every continuous action of $G$ on a compact space has a fixed point.
In \cite{Fraisse_lim_Ramsey}, Kechris, Pestov, and Todorcevic discovered a very general correspondence
which equates the extreme amenability of the automorphism group of an ordered Fra\"iss\'e structure
with the Ramsey Property of its finite substructures.
\begin{thm} \cite{Fraisse_lim_Ramsey}
Let $G$ be a closed subgroup of $S_\infty$.
The following are equivalent:
\begin{enumerate}

\item $G$ is extremely amenable.

\item $G = \Aut (\Gfrak)$ where $\Gfrak$ is a Fra\"iss\'e structure with
an order relation and the finite substructures of $\Gfrak$ have the Ramsey Property.

\end{enumerate}
\end{thm}
At the time of \cite{Fraisse_lim_Ramsey}, it was unclear whether
there was an analogous connection between amenability and
Ramsey theory.
In Section \ref{KPT_section} it will be shown that such an analogous result does exist.

The notation will be mostly standard.
Following a set-theoretic convention, I will sometimes abbreviate
$\{0,\ldots,k-1\}$ with $k$.
The set of natural numbers is taken to include $0$ and all counting will begin at $0$.
The letters $i,j,k,l,m,n$ will be used to denote natural numbers unless otherwise
stated.


\section{A Ramsey theoretic criterion for amenability}

In this section we will prove most of Theorem \ref{amen_criterion},
deferring the equivalence of amenability with
(\ref{0-ramsey:intro}) to the next section.
Before we begin, it will be necessary to extend the evaluation
map $(\nu,E) \mapsto \nu(E)$ to a bilinear map on $P(G) \times \ell^\infty (G)$ by
integration.
We will only need this for finitely supported $\nu$ in which case
\[
\nu (f) = \sum_{g \in B} \nu (\{g\}) f(g)
\]
We will define $f(\nu) = \nu (f)$.
Observe that the map $(\nu,f) \mapsto \nu (f)$ is bilinear.

When proving the theorem, it will be natural to further divide the
task as follows.

\begin{thm} \label{amen_criterion2}
Let $G$ be a group and $H \subseteq G$ be closed under the operation
and contains the identity of $G$.
The following are equivalent:
\begin{enumerate}

\item \label{single_amen}
For every $E \subseteq G$ and every finite $A \subseteq H$,
there is a $\mu$ in $P(H)$ such that
$\mu(g^{-1} E) = \mu(E)$ for every $g \in A$.

\item \label{ramsey}
For every finite $A \subseteq H$, there is a finite $B \subseteq H$ such that $B$
is $\frac{1}{2}$-Ramsey with respect to $A$.

\item \label{ramsey_f3/4}
There is a positive $q < 1$ such that for every finite $A \subseteq H$, there is a finite $B \subseteq H$ such that if
$f:B \to [0,1]$, then there is a $\nu$ in $P(B)$ such that for all $g,g' \in A$,
$g    \nu$ is in $P(B)$ and
\[
|g    \nu (f) - g'    \nu (f) | \leq q.
\]

\item \label{ramsey_f}
For every finite $A \subseteq H$ and $\epsilon > 0$, there is a finite $B \subseteq H$ such that if
$f:B \to [0,1]$ then there is a $\nu$ in $P(B)$ such that for all $g,g' \in A$,
$g    \nu$ is in $P(B)$ and
\[
|g    \nu (f) - g'    \nu (f) | < \epsilon.
\]

\item \label{amenable}
There is a $\mu \in P(H)$ such that for every $E \subseteq G$,
$\mu(g^{-1} E) = \mu (E)$ whenever $g$ is in $H$.

\end{enumerate}
\end{thm}

\begin{proof}
Observe that trivially (\ref{amenable}$\Rightarrow$\ref{single_amen}).
It is therefore sufficient to prove
(\ref{single_amen}$\Rightarrow$\ref{ramsey}$\Rightarrow$\ref{ramsey_f3/4}$\Rightarrow$\ref{ramsey_f}$\Rightarrow$\ref{amenable}).

(\ref{single_amen}$\Rightarrow$\ref{ramsey}):
Suppose that (\ref{ramsey}) is false for some finite $A \subseteq H$.
I claim there is a set $E \subseteq H$ such that for every $\mu \in P(H)$, there are $g,h \in A$
such that $|\mu(g^{-1} E) - \mu (h^{-1} E)| > \frac{1}{2}$ --- a condition which implies the
failure of (\ref{single_amen}).
By replacing $G$ by a subgroup if necessary, we may assume that $G$ is generated by $A$ and
in particular that $G$ is countable.
Let $B_n$ $(n < \infty)$ be an increasing sequence of finite sets covering $H$.
Define $T_n$ to be the collection of all pairs $(n,E)$ where $E$ is a subset of $B_n$
which witness that $B_n$ is not $\frac{1}{2}$-Ramsey with respect to $A$.
Let $T = \bigcup_n T_n$ and if $(n,E)$ and $(n',E')$ are in $T$,
define $(n,E) <_T (n',E')$ if $n < n'$ and $E = E' \cap B_n$.
Observe that if $(n',E')$ is in $T_{n'}$ and $n < n'$, then $(n,E' \cap B_n)$ is in $T_n$.
Thus $(T,<_T)$ is an infinite, finitely branching tree and hence there is an $E \subseteq H$ such that
$(n,E \cap B_n)$ is in $T_n$ for each $n$.
If there were a measure $\nu$ such that $|\nu(g^{-1} E) - \nu (h^{-1}  E)| < \frac{1}{2}$
for all $g,h \in A$,
there would exist such a $\nu$ which has a finite support $S$.
But this would be a contradiction since then $S \cup (A \cdot S)$
would be contained in some $B_n$
and would witness that $(n,E \cap B_n)$ was not in $T_n$.

(\ref{ramsey}$\Rightarrow$\ref{ramsey_f3/4}):
Let $A \subseteq H$ be a given finite set
and let $B \subseteq H$ be finite and $\frac{1}{2}$-Ramsey with respect to $A$.
It suffices to prove that $B$ satisfies the conclusion of (\ref{ramsey_f3/4}) with $q = 3/4$.
Let $f:B \to [0,1]$ be given and define $E = \{b \in B : f(b) \geq 1/2\}$.
By assumption, there is a $\nu$ in $P(B)$ such that $P(A)  \nu \subseteq P(B)$
and for all $g,g' \in A$,
$|g  \nu (E) - g'  \nu (E)| \leq 1/2$.
Also
\[
0 \leq \min(g  \nu (f - \frac{1}{2} \chi_E), g'  \nu (f- \frac{1}{2}\chi_E))
\]
\[
\max(g  \nu (f - \frac{1}{2}\chi_E), g'  \nu (f- \frac{1}{2}\chi_E)) \leq 1/2
\]
Notice that if $0 \leq a,b \leq 1/2$, then $|a-b| \leq 1/2$.
Therefore for all $g,g' \in A$
\[
|g  \nu (f) - g'  \nu (f)| =
\]
\[
|\frac{1}{2}(g  \nu (E) - g'   \nu(E)) + g  \nu(f - \frac{1}{2}\chi_E) - g'  \nu (f-\frac{1}{2}\chi_E)|
\]
\[
\leq \frac{1}{2} |g  \nu (E) - g'  \nu (E)| + |g  \nu(f- \frac{1}{2}\chi_E) - g'  \nu (f-\frac{1}{2}\chi_E)|
\]
\[
 \leq 1/4 + 1/2.
\]

(\ref{ramsey_f3/4}$\Rightarrow$\ref{ramsey_f}):
Let $A \subseteq H$ and $\epsilon >0$ be given.
Let $n$ be such that $q^n < \epsilon$ and construct a sequence
$B_i$ $(i \leq n)$ such that, setting $B_0 = A$,
$B_{i+1}$ satisfies the conclusion of (\ref{ramsey_f3/4}) with respect to 
$B_i$ and
\[
B_i \cup (B_i \cdot B_i) \subseteq B_{i+1}.
\]
Construct $\nu_i$ $(i < n)$ by downward recursion
such that $P(B_i)  \nu_i \subseteq P(B_{i+1})$ and for all
$g,g' \in B_i$,
\[
| g  \nu_i  \cdots  \nu_{n-1} (f) - g'  \nu_i \cdots  \nu_{n-1}( f)| \leq q^{n-i}
\]
This is achieved by applying (\ref{ramsey_f3/4}) to the function $f_i$ defined by
\[
f_{n-1} = f
\]
\[
f_i(g) =  (1/q)^{n-i-1} \Big( g \nu_{i+1} \cdots  \nu_{n-1} (f) - r_i
 \min_{g' \in A}\ g'\nu_{i+1}  \cdots  \nu_{n-1} (f) \Big)
\]
if $i < n-1$.
Our inductive hypothesis implies that
the range of $f_i$ is contained within $[0,1]$.
Therefore there is a $\nu_i$ such that
$P(B_i)  \nu_i \subseteq P(B_{i+1})$ and
\[
|g  \nu_i (f_i) - g'  \nu_i (f_i)| \leq q
\]
holds for every $g,g' \in B_i$ 
and thus
\[
(1/q)^{n-i-1} |g  \nu_i  \cdots  \nu_{n-1} (f) - g'  \nu_i  \cdots  \nu_{n-1} (f)|
\]
\[
= (1/q)^{n-i-1} |f(g  \nu_i  \cdots  \nu_{n-1}) - f(g  \nu_i   \cdots  \nu_{n-1})|
\]
\[
= |f_i(g  \nu_i) - f_i(g'  \nu_i)| = |g  \nu_i (f_i) - g'  \nu_i (f_i)| \leq q.
\]
Multiplying both sides of the inequality by $q^{n-i-1}$, we see that
$\nu_i$ satisfied the desired inequality.
This completes the recursion.
If $\nu = \nu_0  \cdots  \nu_{n-1}$, then for all
$g,g'$ in $A = B_0$ we have that
\[
|g  \nu (f) - g'  \nu (f)| \leq q^n < \epsilon.
\]

(\ref{ramsey_f}$\Rightarrow$\ref{amenable}):
Observe that, by compactness, it is sufficient to prove that for every $\epsilon > 0$, every
finite list $E_i$ $(i < n)$ of subsets of $H$, and
$g_i$ $(i < n)$ in $H$, there is a finitely supported $\mu \in P(H)$ such that
$$
|\mu (g_i^{-1}  E_i) - \mu (E_i) | < \epsilon.
$$
Set $B_0 = \{e_G\} \cup \{g_i : i < n\}$ and construct a sequence
$B_i$ $(i \leq n)$ such that $B_i \cdot B_i \subseteq B_{i+1}$ and
$B_{i+1}$ satisfies (\ref{ramsey_f}) with $B_i$ in place
of $A$ and $\epsilon/2$ in place of $\epsilon$.
Inductively construct $\nu_i$ $(i < n)$ by downward recursion on $i$.
If $\nu_j$ $(i < j < n)$ has been constructed, let $\nu_i \in P(B_{i+1})$ be such that
$P(B_i)  \nu_i \subseteq P(B_{i+1})$ and
$$
|\mu  \nu_{i}  \cdots  \nu_{n-1} (E_i) - \mu' \nu_i  \cdots  \nu_{n-1} (E_i)| < \epsilon/2
$$
for all $\mu , \mu' \in P(B_i)$.
Set $\mu = \nu_0  \cdots  \nu_{n-1}$.
If $i < n$, then since $\nu_0  \cdots  \nu_{i-1}$ and
$g_i^{-1}  \nu_0 \cdots \nu_{i-1}$ are in $P(B_i)$,
$$
|g_i  \mu (E_i) - \nu_i  \cdots  \nu_{n-1} (E_i)| < \epsilon/2
$$
$$
|\mu (E_i) - \nu_i  \cdots  \nu_{n-1} (E_i)| < \epsilon/2
$$
and therefore
$|\mu (g_i^{-1}  E_i) - \mu (E_i)| < \epsilon$.
\end{proof}

\section{Comparing the Ramsey and F\o lner functions}

\label{Ramsey_function:sec}

The purpose of this section is to define the \emph{Ramsey function} of a finitely generated
group with respect to a finite generating set and relate it to the \emph{F\o lner function}
which has been studied in, e.g., \cite{iso_profiles_fg}, \cite{pw_auto_groups}, \cite{entropy_isoperimetry}.
The main result of this section is due to Henry Towsner, answering a question in an
early draft of this paper:
\emph{The F\o lner function for a given group and generating set
can be obtained from the Ramsey function by primitive recursion.}
It is included with his kind permission.

We will now turn to the definitions of the F\o lner and Ramsey functions.
Let $G$ be a group with a fixed finite generating set $S$ (which is not
required to be closed under inversion).
Let $B_n$ denote the elements of $G$ whose distance from the identity is at most $n$
in the word metric.
Define the following functions:
\begin{itemize}

\item
$\Fol_{G,S} (k)$ is the minimum cardinality of a $1/k$-F\o lner set with respect to
the generating set $S$.

\item
$F_{G,S}(m,\epsilon)$ is the minimum $n$ such that there is a $\nu$ in $P(B_n)$
such that $P(B_m) \nu \subseteq P(B_n)$ and
$\sum_{g \in B_m} ||g \nu - \nu||_{\ell_1} < \epsilon$.

\item
$R_{G,S}(m,\epsilon)$ is the minimum $n$ such that $B_n$ is
$\epsilon$-Ramsey with respect to $B_m$.

\item
$\tilde R_{G,S}(m,\epsilon,l)$ is the minimum $n$ such that if $f_i$ $(i < l)$ is a sequence of functions
from $B_n$ into $[0,1]$, then there is a $\nu \in P(B_n)$ such that $P(B_m)  \nu \subseteq P(B_n)$
and such that for every $g,g' \in B_m$ and $i < l$,
\[
|g  \nu (f_i) - g'  \nu (f_i)| < \epsilon
\]

\end{itemize}
The definition of $F_{G,S}$ is formulated so that it
is a triviality that $R_{G,S}(m,k) \leq \tilde R_{G,S}(m,k) \leq F_{G,S}(m,k)$
holds for all $m$ and $k$.
The following relationship holds between $F_{G,S}$ and $\Fol_{G,S}$:
\[
\Fol_{G,S} (k) \leq {(2|S|+1)}^{F_{G,S} (1,1/k)}
\]
The reason for this is that the $n$-ball in $G$ with respect to $S$ contains
at most $(2|S|+1)^n$ elements and if $\nu \in \ell^1(G)$ is such that
\[
\sum_{g \in S} ||g \nu - \nu||_{\ell^1} < \epsilon
\]
then the support of $\nu$ contains an $\epsilon$-F\o lner set with respect to $S$ \cite{Folner:Namioka}.

Set $R_{G,S}(m) = R_{G,S}(m,1/2)$ and $\tilde R_{G,S}(m,\epsilon) = \tilde R_{G,S}(m,\epsilon,1)$.
The proof of Theorem \ref{amen_criterion2} shows that 
\[
\tilde R_{G,S}(m,\epsilon) \leq R_{G,S}^p(m)
\]
whenever $(3/4)^p < \epsilon$ (here $R^p_{G,S}$ denotes the $p$-fold composition of
$R_{G,S}$).
Furthermore, it shows that
\[
\tilde R_{G,S}(m,\epsilon,l) \leq \tilde R_{G,S}(\tilde R_{G,S}(m,\epsilon,l-1),\epsilon) 
\]
\[
= \tilde R_{G,S}(\tilde R_{G,S}( \ldots \tilde R_{G,S}(m,\epsilon) \ldots,\epsilon),\epsilon)
\leq R^{lp}_{G,S}(m)
\]
whenever $l > 1$.
Finally we have the following proposition.
\begin{prop}
$F_{G,S} (m,2 \epsilon |S|) \leq \tilde R_{G,S}(m,\epsilon,|S|)$.
\end{prop}

\begin{proof}
Let $B = B_n$ where $n = \tilde R_{G,S}(m,\epsilon,|S|)$.
Define
\[
C = \{\Seq{ g    \nu - \nu : g \in S} : \nu \in P(B) \textrm{ and } P(S) \nu \subseteq P(B) \}.
\]
\[
U = \{\xi \in (\ell^1(B))^S : \sum_{g\in S} ||\xi_g||_{\ell^1} < 2 \epsilon |S| \}
\]
Observe that $C$ and $U$ are both convex subsets
of $(\ell^1(B))^S$ with $C$ being compact and $U$ being open.
If $C \cap U$ is non empty, then there is a $\nu$ in $P(B)$ such that $P(S) \nu \subseteq P(B)$ and
\[
\sum_{g \in S} ||g \nu - \nu||_{\ell^1} < 2 \epsilon |S|.
\]
In particular, we would have that
$F_{G,S}(m,2 \epsilon |S|) \leq n =  \tilde R_{G,S}(m,\epsilon,|S|)$.

Now suppose for contradiction that $C$ and $U$ are disjoint.
By the Hahn-Banach separation theorem (see \cite[3.4]{functional_analysis:Rudin}),
there is a linear functional $\Lambda$ defined on $(\ell^1(B))^S$ such that, for some $r \in \Rbb$, 
$\Lambda \xi < r$ if $\xi \in U$ and $r \leq \Lambda \xi$ if
$\xi \in C$.
In the present setting, such a functional $\Lambda$ takes the form
\[
\Lambda \xi = 
\sum_{g \in S} \xi_g f_g
\]
for some $\Seq{f_g : g \in S} \in (\ell^\infty(B))^S$.
If we give $(\ell^1(B))^S$ the norm by identifying it with $\ell^1(B \times S)$,
then we may assume that $\Lambda$ has norm $1$.
Since $\ell^1(B \times S)^*$ is isometric to $\ell^\infty(B \times S)$, it follows
that $|f_g(b)| \leq 1$ for all $b$ and $g$ with equality obtained for some $(b,g) \in B \times S$.
It follows that we may take $r = \epsilon |S|$.
This is a contradiction, however, since by our choice of $B = B_n$,
there is a $\nu$ in $P(B)$ such that $P(S) \nu \subseteq P(B)$ and for all $g \in S$,
\[
|g \nu (f_g) - \nu (f_g)| < 2\epsilon
\]
(the factor of 2 is because
$f_g$ maps into an interval of length $2$)
and therefore $\sum_{g \in S} |g \nu (f_g) - \nu (f_g)| < 2 \epsilon |S|$.
\end{proof}

Putting this together, we have the following upper bound on the F\o lner function
in terms of the iterated Ramsey function.

\begin{thm}
$\Fol_{G,S} (k) \leq (2s + 1)^{R^{ps}(1)}$ whenever
$(3/4)^p < 1/(2ks)$ where $s = |S|$.
\end{thm}


\section{Invariantly measurable sets in $\Fbb_2$}

In light of the theorem of the previous section, it is natural to define, for
an arbitrary group $G$, the collection $\Mscr_G$ of subsets of $G$ which are
invariantly measurable.
It is tempting to suspect that Theorem \ref{amen_criterion}
might be subsumed in a more general result
which asserts that, in any group $G$, there is a $\mu$
which measures each element of $\Mscr_G$ invariantly.
Theorem \ref{inv_meas_c.e.}, whose proof we now turn to, asserts that this is not the case.

\begin{proof}
Let $a$ and $b$ denote the generators of $\Fbb_2$ and
let $A$ denote the collection of all elements of $\Fbb_2$ whose reduced word begins
with $a$ or $a^{-1}$.
Let $h:\Fbb_2 \to \Zbb$ be the homomorphism which sends $a$ to $1$ and $b$ to $-1$
and define $Z_k = \{w \in \Fbb_2 : h(w) > k\}$, setting $Z = Z_0$.
Define
\[
X = A \cup Z^{\C} \quad \quad X' = A \cap Z = X \setminus Z^{\C}
\]
\[
Y = A^{\C} \cup Z \quad \quad Y' = A^{\C} \cap Z^{\C} = Y \setminus Z
\]
(Here a superscript of $\C$ denotes complementation.)
First we will show that $X$, $X'$, $Y$, $Y'$, and $Z$ are each invariantly measurable.
Observe that $X' \subseteq Z \subseteq Y$ and $Y' \subseteq Z^{\C} \subseteq X$.
It therefore suffices to find
measures $\mu_0$ and $\mu_1$ such that $\mu_i(w    Z) = i$ for $i = 0,1$ and $w \in \Fbb_2$.
This is because then $\mu_0$ will measure $X$, $X'$, and $Z$ invariantly
(with measures $1$, $0$, and $0$) and
$\mu_1$ will measure $Y$, $Y'$, and $Z$ invariantly
(with measures $1$, $0$, and $1$).
Such measures are constructed
by extending the families $\{Z_k^{\C} : k \in \Zbb\}$
and $\{Z_k : k \in \Zbb\}$, each of which have the finite intersection property,
to ultrafilters, and regarding them as elements of $P(\Fbb_2)$.
Invariance follows from the observation that $w    Z = Z_{h(w)}$ and the containments noted above.

Now suppose for contradiction that there is a
$\mu \in P(\Fbb_2)$ which measures $X$, $X'$, $Y$, $Y'$, and $Z$ invariantly.
Since the sequences $a^k    Y'$ $(k < \infty)$ and $b^k    X'$ $(k < \infty)$ each consist
of pairwise disjoint elements, $\mu(X') = \mu (Y') = 0$.
Observe that $A \symdif Z^{\C} = X' \cup Y'$.
Since $X'$ and $Y'$ are measured invariantly and are disjoint,
$X' \cup Y'$ is also measured invariantly.
Therefore we have
\[
\mu ((b^k    A) \symdif (b^k    Z^{\C})) = 
\mu (b^k    (A \symdif Z^{\C})) = 
\mu (A \symdif Z^{\C}) = 0.
\]
Thus $\mu (b^k    A) = \mu (b^k    Z^{\C}) = \mu (Z^{\C})$.
Since $b^k   A$ $(k < \infty)$ is a sequence of pairwise disjoint sets,
$\mu (Z^{\C}) = 0$.
Similarly $A^{\C} \symdif Z = X' \cup Y'$ and therefore by a similar argument
$\mu (a^k    A^{\C}) = \mu (Z)$ for all $k$.
Using the fact that
$a^k    A^{\C}$ $(k < \infty)$ is a sequence of pairwise disjoint sets one obtains that
$\mu (Z) = 0$.
This is a contradiction since
$\mu (Z) + \mu (Z^{\C}) = 1$.
This finishes the proof.
\end{proof}

\section{A criterion for non amenability: unbalanced puzzles}

The purpose of this section is to prove Theorem \ref{puzzle}.
The following two simple propositions capture most of what is left to prove.

\begin{prop} \label{balanced_to_measure}
Let $G$ be a group, $\epsilon \geq 0$, and $A$ be a finite subset of $G$.
If $E \subseteq G$, $\Yscr$ is $\epsilon$-balanced, and $B$ is a finite
set such that 
\[
\Yscr = \{X_E (g) : g \in B\},
\]
then there is a $\nu \in P(B)$ such that
$|a    \nu (E) - a'    \nu(E)| \leq \epsilon$
for all $a,a' \in A$.
\end{prop}

\begin{remark}
Notice that a typical $B$ satisfying the hypothesis of this proposition
may well satisfy that it is its own boundary in the Cayley graph, even if
$\epsilon = 0$.
This is again quite different than what is possible with F\o lner sets
(even if the F\o lner sets are allowed to be ``weighted'').
\end{remark}

\begin{proof}
Let $\mu \in P(\Yscr)$ be such that
\[
|\mu (\{X \in \Yscr : a \in X\}) - \mu (\{X \in \Yscr : a' \in X\})| \leq \epsilon
\]
for every $a,a' \in A$.
By replacing $B$ be a subset, we may assume that for each $b \ne b'$ in $B$,
$X_E(b) \ne X_E(b')$.
Define $\nu \in P(B)$ by
$\nu (\{b\}) = \mu (\{X_E(b)\})$.
Now suppose that $a \in A$.
\[
\nu (a^{-1}    E) = 
\sum \{\nu (\{b\}) : b \in a^{-1}    E \} 
\]
\[
= \sum \{\mu (\{X_E(b)\}) : a    b \in E\} =
\mu (\{X \in \Yscr : a \in X\})
\]
The conclusion now follows from our choice of $\mu$.
\end{proof}

\begin{prop} \label{measure_to_balanced}
Let $G$ be a group and $A$ be a finite subset of $G$.
If $E \subseteq G$ and there is a $\nu \in P(G)$ such that
\[
|a    \nu(E) - a'    \nu (E)| \leq \epsilon
\]
for every $a,a' \in A$,
then 
\[
\{X \in \Xscr_E : \nu (\{g \in G : X_E (g) = X\}) > 0 \}
\]
is $\epsilon$-balanced (and in particular $\Xscr_E$ is
$\epsilon$-balanced).
\end{prop}

\begin{proof}
Let $G$, $A$, $E$, and $\nu$ be given as in the statement of the proposition.
For each $X$ in $\Xscr_E$, define
\[
\mu(\{X\}) = \nu (\{g \in G : X_E(g) = X\}).
\]
It is sufficient to show that if $a$ is in $A$,
then $\sum_{X \ni a} \mu(\{X\}) = \nu (a^{-1}    E)$.
To this end
\[
\sum_{X \ni a} \mu(\{X\}) = \nu (\{g \in G : a \in X_E(g)\})
\]
\[
= \nu (\{g \in G : a  g \in E\}) = \nu (a^{-1}    E) .
\]
\end{proof}

Now we are ready to prove Theorem \ref{puzzle}.
All implications will be established by proving the contrapositive.
The implication (\ref{non_amen:intro}$\Rightarrow$\ref{unbalanced:intro}) follows from
Proposition \ref{balanced_to_measure} together with the
equivalence of (\ref{inv_measure:intro}) and (\ref{amen:intro}) in Theorem \ref{amen_criterion}.
The implication (\ref{unbalanced:intro}$\Rightarrow$\ref{no_0-ramsey:intro})
is given by Proposition \ref{measure_to_balanced}.

Finally, in order to see the implication (\ref{no_0-ramsey:intro}$\Rightarrow$\ref{non_amen:intro}),
suppose that $G$ is amenable and $A \subseteq G$ is finite.
Let $\epsilon > 0$ be such that if $\Yscr$ is a collection of subsets of $A$ which is $\epsilon$-balanced,
then $\Yscr$ is $\epsilon'$-balanced for all $\epsilon' > 0$.
This is possible since the collection of all families of subsets of $A$ is finite.
Let $B$ be $\epsilon$-Ramsey.
It suffices to prove that $B$ is $\epsilon'$-Ramsey for each $\epsilon' > 0$ since it then follows by
compactness that $B$ is $0$-Ramsey.
Suppose that $E \subseteq B$.
By our assumption on $B$, there is a $\nu \in P(B)$ such that
$A    \nu \subseteq P(B)$ and such that
\[
|g    \nu (E) - g'    \nu(E)| \leq \epsilon
\]
for all $g,g' \in A$.
By Proposition \ref{measure_to_balanced},
\[
\Yscr = \{X_E(g) : (g \in B) \land (A g \subseteq B)\}
\]
is $\epsilon$-balanced.
By assumption this collection is $\epsilon'$-balanced for every $\epsilon' > 0$.
Therefore by Proposition \ref{balanced_to_measure},
there is a $\nu$ in $P(B)$ such that $A    \nu \subseteq P(B)$ and
\[
|g    \nu (E) - g'    \nu(E)| \leq \epsilon'.
\]
for all $g,g' \in A$.
This finishes the proof of the theorem.

\section{Structural Ramsey theory and KPT Theory}

\label{KPT_section}

In this section I will place the results of the present paper into
the context of the theory of Kechris, Pestov, and Todorcevic developed in
\cite{Fraisse_lim_Ramsey} which equates the property of \emph{extreme amenability}
of certain automorphism groups to \emph{structural Ramsey theory}.
First we will need to recall some notation and terminology
from \cite{Fraisse_lim_Ramsey}; further reading can be found there.
A \emph{Fra\"iss\'e structure} is a countable relational structure $\Afrak$
which is \emph{ultrahomogeneous}
--- every finite partial automorphism extends to an automorphism of the whole structure.
If moreover $\Afrak$ includes a relation which is a linear order, then $\Afrak$ is said to be a
\emph{Fra\"iss\'e order structure}.
Some notable examples of such structures are $(\Qbb,\leq)$,
the \emph{random graph}, and \emph{rational Urysohn space}.
If $G$ is a countable group, then we may also associate to $G$ the Fra\"iss\'e structure
$\Gfrak = (G; R_g : g \in G)$
where
\[
R_g = \{ (a,b) \in G^2 : a    b^{-1} = g \}
\]
Observe that the automorphisms of $\Gfrak$ are given by right translation and therefore
$\Aut(\Gfrak) \iso G$.
Since every automorphism of $\Gfrak$ is determined by where is sends the identity,
$\Aut (\Gfrak)$ is discrete as a subgroup of the group of all permutations of $G$ equipped
with the topology of point-wise convergence.

If $\Afrak$ is a Fra\"iss\'e (order) structure, then $\Age(\Afrak)$ is the collection of finite substructures of $\Afrak$.
A collection arising in this way is called a \emph{Fra\"iss\'e (order) class}.
It should be noted that Fra\"iss\'e (order) classes have an intrinsic axiomatization, although this will not
be relevant for the present discussion.

If $\Cscr$ is a Fra\"iss\'e class and $\Bfrak$ and $\Afrak$ are structures in $\Cscr$, then
let $\binomial{\Bfrak}{\Afrak}$ denote the collection of all embeddings of $\Afrak$ into $\Bfrak$.
Define $\Cfrak \rightarrow (\Bfrak)^{\Afrak}_k$ if whenever $f:\binomial{\Cfrak}{\Afrak} \to k$,
there is a $\beta$ in $\binomial{\Cfrak}{\Bfrak}$ such that $f$ is constant on
$\binomial{\beta}{\Afrak} = \{\beta \circ \alpha : \alpha \in \binomial{\Bfrak}{\Afrak}\}$.
A Fra\"iss\'e class $\Cscr$ has the \emph{Ramsey Property} if for every $\Afrak$ and $\Bfrak$ in $\Cscr$,
there is a $\Cfrak$ in $\Cscr$ such that $\Cfrak \rightarrow (\Bfrak)^{\Afrak}_2$.

The main result of \cite{Fraisse_lim_Ramsey}
is that, for a Fra\"iss\'e order structure $\Afrak$,
$\Aut(\Afrak)$ is extremely amenable if
and only if $\Age(\Afrak)$ has the Ramsey Property.
The power of this theorem comes from the rich literature on the Ramsey Property of Fra\"iss\'e classes.
That the finite linear orders form a Ramsey class is just a reformulation of the finite
form of Ramsey's theorem.
More sophisticated examples are the classes of finite ordered graphs \cite{part_fin_rel_set_sys} \cite{ramsey_class_set_sys}, 
finite naturally ordered Boolean algebras \cite{dual_ramsey},
finite ordered metric spaces \cite{metric_ramsey},
and finite dimensional
naturally ordered vector spaces
over a finite field \cite{ramsey_class_cat}.
The branch of mathematics concerned with such results
is known as \emph{structural Ramsey theory}.

If $G$ is a countable group, 
the collection $\Gscr$ of finite substructures of $\Gfrak$ will never form a Ramsey class.
One reason for this is the result of Veech \cite{top_dyn:Veech} asserting that a locally compact group
can never be extremely amenable.
In the case of finitely generated groups, this can be seen explicitly:
the functions $f:\binomial{\Gfrak}{\efrak} \to 2$ defined by
\[
f(g) \equiv d_S(e,g) \mod 2
\]
where $S$ is a generating set for $G$, $d_S$ is the word metric, and $\efrak$ is $\{e\}$ regarded as a
substructure of $\Gfrak$.
(Observe that $\binomial{\Gfrak}{\efrak}$ can naturally be identified with the singletons in $G$.)

We can however modify the Ramsey Property as follows.
Let $\Afrak$ and $\Bfrak$ be finite substructures of a relational structure $\Xfrak$.
Define $\cbinomial{\Bfrak}{\Afrak}$ to be the collection of all finitely supported probability measures
on $\binomial{\Bfrak}{\Afrak}$.
If $f:\binomial{\Bfrak}{\Afrak} \to \Rbb$, then $f$ extends to a linear function defined
on the vector space generated by $\binomial{\Bfrak}{\Afrak}$; this extension will also
be denoted by $f$.
Extending $\circ$ bilinearly, we define
$\binomial{\beta}{\Afrak}$ and $\cbinomial{\beta}{\Afrak}$ when
$\beta$ is in $\cbinomial{\Xfrak}{\Bfrak}$.

Define $\Cfrak \rightarrow \langle \Bfrak \rangle^{\Afrak}_k$ to mean that
whenever $f:\binomial{\Cfrak}{\Afrak} \to k$,
there is a $\beta \in \cbinomial{\Cfrak}{\Bfrak}$ such that
if $\alpha,\alpha' \in \cbinomial{\beta}{\Afrak}$,
\[
|f(\alpha) - f(\alpha')| \leq 1/2.
\]
It follows from the definitions that if $A$ and $B$ are finite subsets of a group $G$,
then $\Bfrak \rightarrow \langle \Afrak \rangle^{\efrak}_2$ is equivalent to asserting
that $B$ is $1/2$-Ramsey with respect to $A$.
Therefore, by Theorem \ref{amen_criterion}  the amenability of $G$
is equivalent to the following \emph{convex Ramsey property} of $\Gscr$:
for every $\Afrak$ and $\Bfrak$ in $\Gscr$ there is a $\Cfrak$ in $\Gscr$ such that
$\Cfrak \rightarrow \langle \Bfrak \rangle^{\Afrak}_2$.
The purpose of the remainder of this section is to prove the following generalization of
Theorem \ref{amen_criterion2} to the setting of automorphism groups of Fra\"iss\'e structures.

\begin{thm}
If $\Xfrak$ is a Fra\"isse structure, then the following are equivalent:
\begin{enumerate}

\item \label{conv_RP_inf}
for every $\Afrak$ and $\Bfrak$ in $\Age (\Xfrak)$,  and every
$f:\binomial{\Xfrak}{\Afrak} \to \{0,1\}$, there is a $\beta$
in $\cbinomial{\Xfrak}{\Bfrak}$ such that for every $\alpha,\alpha' \in \cbinomial{\beta}{\Afrak}$,
$|f(\alpha) - f(\alpha')| \leq 1/2$.

\item \label{conv_RP}
$\Age (\Xfrak)$ satisfies the convex Ramsey property:
for every $\Afrak$ and $\Bfrak$ in $\Age (\Xfrak)$ there is a $\Cfrak$ in $\Age(\Xfrak)$ such
that $\Cfrak \rightarrow \langle \Bfrak \rangle^{\Afrak}_2$.

\item \label{conv_RP_fp}
there is a $p  < 1$ such that
for every $\Afrak$ and $\Bfrak$ in $\Age (\Xfrak)$ there is a $\Cfrak$ in $\Age(\Xfrak)$
such that for every $f:\binomial{\Cfrak}{\Afrak} \to [0,1]$, there is a $\beta$
in $\cbinomial{\Cfrak}{\Bfrak}$ such that for every $\alpha,\alpha' \in \cbinomial{\beta}{\Afrak}$,
\[
|f(\alpha) - f(\alpha')| \leq p.
\]

\item \label{conv_RP_f}
for every $\Afrak$ and $\Bfrak$ in $\Age (\Xfrak)$,  every $\epsilon > 0$,
there is a $\Cfrak$ in $\Age(\Xfrak)$ such that for every
$f:\binomial{\Cfrak}{\Afrak} \to [0,1]$, there is a $\beta$
in $\cbinomial{\Cfrak}{\Bfrak}$ such that for every $\alpha,\alpha' \in \cbinomial{\beta}{\Afrak}$,
\[
|f(\alpha) - f(\alpha')| \leq \epsilon.
\]

\item \label{conv_nRP_f}
for every $\Afrak$ and $\Bfrak$ in $\Age (\Xfrak)$,  every $\epsilon > 0$, and $n$,
there is a $\Cfrak$ in $\Age(\Xfrak)$ such that for every
sequence $f_i$ $(i < n)$ of functions from $\binomial{\Cfrak}{\Afrak}$ to $[0,1]$,
there is a $\beta$ in $\cbinomial{\Cfrak}{\Bfrak}$ such that for every $\alpha$ in
$\cbinomial{\beta}{\Afrak}$,
\[
|f_i(\alpha) - f_i(\alpha')| \leq \epsilon.
\]

\item \label{aut_amen}
$\Aut (\Xfrak)$ is amenable.

\end{enumerate}
\end{thm}

\begin{remark}
The equivalence of (\ref{conv_nRP_f}) and (\ref{aut_amen}) was noticed by
Todor Tsankov, prior to the results of this paper.
I would like to thank him for a helpful conversation in which it became clear
that the above theorem should be true.
\end{remark}

\begin{proof}
I will only prove the implications (\ref{conv_RP_inf}$\Rightarrow$\ref{conv_RP}),
(\ref{conv_RP_f}$\Rightarrow$\ref{conv_nRP_f}), (\ref{conv_nRP_f}$\Rightarrow$\ref{aut_amen}), and
(\ref{aut_amen}$\Rightarrow$\ref{conv_RP_inf}).
The remaining implications are only notationally different from their counterparts in Theorem \ref{amen_criterion2} and
the implications which will be proved will demonstrate how these notational adaptations are made.

To see (\ref{conv_RP_inf}$\Rightarrow$\ref{conv_RP}), we will suppose that (\ref{conv_RP}) is false and
prove that (\ref{conv_RP_inf}) is false.
To this end, let $\Afrak$ and $\Bfrak$ be given.
Let $X$ be the underlying set for the structure $\Xfrak$ and let
$X_n$ $(n < \infty)$ be an increasing sequence of finite sets whose union is $X$.
For each $n$, fix a $f_n:\binomial{\Xfrak_n}{\Afrak} \to 2$ such that there is
no $\beta \in \cbinomial{\Xfrak_n}{\Bfrak}$ such that for all $\alpha,\alpha' \in \cbinomial{\beta}{\Afrak}$,
$|f(\alpha) - f(\alpha')| \leq 1/2$.
Find a subsequence $f_{n_k}$ $(k < \infty)$ such that for every $m$,
if $k,k' \geq m$, then
\[
f_{n_k} \restriction \binomial{\Xfrak_m}{\Afrak} = f_{n_{k'}} \restriction \binomial{\Xfrak_m}{\Afrak}
\]
Define $f:\binomial{\Xfrak}{\Afrak} \to \{0,1\}$ by
$f(\alpha) = f_{n_k}(\alpha)$ whenever the range of $\alpha$ is contained in $\Xfrak_m$ and
$m \leq k$.
If there were a $\beta$ in $\cbinomial{\Xfrak}{\Bfrak}$ such that
for all $\alpha,\alpha' \in \cbinomial{\beta}{\Afrak}$,
$|f(\alpha) - f(\alpha')| \leq 1/2$, then such a $\beta$ would be contained in $\cbinomial{\Xfrak_m}{\Bfrak}$
for some $m$.
Then for any $k > m$, $\beta$ would contradict our choice of $f_{n_k}$.

In order to see the implication (\ref{conv_RP_f}$\Rightarrow$\ref{conv_nRP_f}),
let $\Afrak$, $\Bfrak$, and $\epsilon > 0$ be given.
Construct $\Cfrak_i$ $(i \leq n)$ such that $\Cfrak_0 = \Bfrak$ and for all $i \leq n$,
if $f:\binomial{\Cfrak_i}{\Afrak} \to [0,1]$, there is a $\nu \in \cbinomial{\Cfrak_i}{\Cfrak_{i-1}}$
such that for all $\alpha,\alpha' \in \cbinomial{\nu}{\Afrak}$,
$|f(\alpha)-f(\alpha')| \leq \epsilon$.
Define $\beta_n = \Cfrak_n$ and
construct $\beta_i$ $(i < n)$ by downward induction such that 
$\beta_i$ is in $\cbinomial{\beta_{i+1}}{\Cfrak_i}$ and if $\alpha,\alpha'\in \cbinomial{\beta_i}{\Afrak}$,
then $|f_i(\alpha) - f_i(\alpha')| < \epsilon$.
This is achieved by applying our hypothesis on $\Cfrak_{i+1}$ to the function
$\tilde f_i : \binomial{\Cfrak_{i+1}}{\Afrak} \to [0,1]$ defined by
$\tilde f_i (\alpha) = f_i(\beta_{i+1} \circ \alpha)$.
If $\nu_i \in \cbinomial{\Cfrak_{i+1}}{\Cfrak_i}$ is such that for all $\alpha,\alpha' \in \cbinomial{\Cfrak_{i+1}}{\Afrak}$
\[
|\tilde f_i (\alpha) - \tilde f_i(\alpha')| \leq \epsilon
\]
then $\beta_i = \beta_{i+1} \circ \nu_i$ is as desired.
Since $i < j < n$ implies $\cbinomial{\beta_i}{\Afrak} \subseteq \cbinomial{\beta_j}{\Afrak}$,
we have that $\beta = \beta_0$ satisfies the conclusion of (\ref{conv_nRP_f}).

Next we will prove (\ref{conv_nRP_f}$\Rightarrow$\ref{aut_amen}).
We will use the following characterization of amenability of a topological group:
$G$ is amenable if and only if whenever $G$ acts continuously on a compact space $K$,
$K$ admits a (countably additive) $G$-invariant Borel probability measure.
To this end, fix a continuous action of $\Aut(\Xfrak)$ on a compact space $K$.
Recall that the Borel probability measures form a weak* compact subset of
$C(K)^*$.
Therefore it is sufficient to prove that for every $\epsilon > 0$,
every sequence $f_i$ $(i < n)$ of elements of $C(K)$,
and every sequence $g_i$ $(i < n)$ of elements of $\Aut(\Xfrak)$,
there is a finitely supported measure
$\nu$ on $X$ such that for every $i < n$
\[
|f_i(g_i \cdot \nu))- f_i(\nu)| \leq \epsilon.
\]
Let $f_i$ $(i < n)$ and $g_i$ $(i < n)$ be given and assume without loss of generality that
$f_i$ maps into $[0,1]$.

By the compactness of $K$, there is an open neighborhood
$U$ of $\id_{\Xfrak}$ such that if $g$ is in $U$, then for all $i < n$
\[
|f_i(g \cdot \nu) - f_i(\nu)| \leq \frac{\epsilon}{2}
\]
(Here we have extended the action linearly to an action of $\Aut(\Xfrak)$ on the
finitely supported measures.
Similarly, elements of $C(K)$ are extended linearly to
the finitely supported measures on $K$.)
Therefore there is a finite substructure $\Afrak$ of $\Xfrak$ such that
if $g \restriction \Afrak = \id_{\Afrak}$, then $g$ is in $U$.
Let $\Bfrak$ be the finite substructure of $\Xfrak$ with
domain
\[
B = A \cup \bigcup_{i < n} g_i^{-1} (A).
\]
Let $\Cfrak$ be the finite substructure of $\Xfrak$ which satisfies the conclusion
of (\ref{conv_nRP_f}) with $\epsilon/2$ in place of $\epsilon$.

Fix an element $x_0$ of $K$.
Observe that if $i < n$ and $g$ and $h$ are in $\Aut(\Xfrak)$
are such that $g^{-1} \restriction \Afrak = h^{-1} \restriction \Afrak$,
then
\[
|f_i(g \cdot x_0) - f_i(h \cdot x_0)| \leq \frac{\epsilon}{2}
\]
This is because otherwise $gh^{-1} \in U$ and $x = h \cdot x_0$ would contradict our
choice of $U$.

For each
$f$ in $C(K)$, define
$\tilde f:\binomial{\Xfrak}{\Afrak} \to [0,1]$ by
\[
\tilde f(\alpha) = \inf \{f(h \cdot x_0) : h \in \Aut(\Xfrak) \land h^{-1} \restriction \Afrak = \alpha \}.
\]
By our choice of $\Cfrak$, there is a $\beta$ in $\cbinomial{\Cfrak}{\Bfrak}$ such
that for every $\alpha,\alpha' \in \cbinomial{\beta}{\Afrak}$ and $i < n$,
\[
|\tilde f_i(\alpha) - \tilde f_i(\alpha')| \leq \frac{\epsilon}{2}.
\]
Let $\beta_j$ $(j < m)$ be the elements of $\binomial{\Cfrak}{\Bfrak}$ such that
for some choice of positive $\lambda_j$ $(j < m)$,
$\beta = \sum_{j < m} \lambda_j \beta_j$.
For each $j < m$, fix an $h_j \in \Aut(\Xfrak)$ such that $h_j$ extends $\beta_j$.
This is possible since $\Xfrak$ is ultrahomogeneous.
Finally,
define
\[
\nu = \sum_{j < m} \lambda_j \delta_{h_j^{-1}  \cdot x_0}
\]
where $\delta_x$ denotes the point mass at $x$.
Define $\alpha_i = g_i^{-1} \restriction \Afrak$,
observing that $\alpha_i \in \binomial{\Bfrak}{\Afrak}$. 
Now for each $i < n$,
\[
|f_i(g_i \cdot \nu) - f(\nu)| =
|f_i(g_i \sum_{j < m} \lambda_j \cdot \delta_{h_j^{-1} \cdot x_0}) - f_i(\sum_{j < m} \lambda_j \delta_{h_j^{-1} \cdot x_0})| 
\]
\[
=|f_i(\sum_{j < m} \lambda_j (g_i \circ h_j^{-1}) \cdot \delta_{x_0}) - f_i(\sum_{j < m} \lambda_j h_j^{-1} \cdot \delta_{x_0})|
\]
\[
\leq |\tilde f_i(\beta \circ \alpha_i) - \tilde f_i (\beta \circ \id_{\Afrak})| + \frac{\epsilon}{2} \leq \epsilon
\]
which is what we needed to prove.

Finally, we will prove (\ref{aut_amen}$\Rightarrow$\ref{conv_RP_inf}).
To this end, let $\Afrak$ and $\Bfrak$ be given and let
$f_0:\binomial{\Xfrak}{\Afrak} \to 2$ be arbitrary.
Observe that $2^{\binomial{\Xfrak}{\Afrak}}$ is a compact space and that
$\Aut(\Xfrak)$ acts continuously on $2^{\binomial{\Xfrak}{\Afrak}}$ on the left by
$g \cdot f(\alpha) = f(g \circ \alpha)$.
Let $Z$ denote the orbit of $f_0$ under this action and let  
$K$ denote the closure of $Z$.
Since $\Aut(\Xfrak)$ is amenable, there is an probability measure $\mu$ on $K$
which is invariant under the action.
Since $\mu$ is invariant, $\int  f(\alpha) d \mu (f)$ does not depend on $\alpha \in \binomial{\Xfrak}{\Afrak}$.

Since the collection of all probabilities measures on $K$ whose support is finite and contained
in $Z$ is dense,
there are $\gamma_j$ $(j < m)$ in $\Aut(\Xfrak)$ and positive $\lambda_j$ $(j < m)$ such that
for each $\alpha \in \binomial{\Bfrak}{\Afrak}$
\[
|\sum_{j < m} \lambda_j f_0 (\gamma_j \circ \alpha)  - r| \leq 1/4.
\]
Now define $\beta_j = \gamma_j \restriction \Bfrak$, $\beta = \sum_{j  < m} \lambda_j \beta_j$ and observe
that
if $\alpha,\alpha' \in \cbinomial{\Bfrak}{\Afrak}$
\[
|f_0 (\beta \circ \alpha) - f_0(\beta \circ \alpha')| \leq |f_0(\beta \circ \alpha) - r| + |f_0(\beta \circ \alpha') - r| \leq 1/2.
\]
\end{proof}

\section{Concluding remarks}

The research presented in this article grew out of a study of the amenability problem for
Thompson's group $F$ and the study of its F\o lner function.
A Ramsey theoretic analysis of the amenability problem for $F$
will be published in a separate article \cite{Hind_Ellis_F}.
In \cite{fast_growth_F}, it was demonstrated that there is a constant $C$ such that
the minimum cardinality of a $C^{-n}$-F\o lner set in $F$ (with respect to the standard
generating set) has cardinality at least $2^{2^{2^{\ldots 2}}}$ (a tower of $n$ 2s).
Theorem \ref{amen_criterion}
was proved in part in hopes that the minimum cardinalities of $\frac{1}{2}$-Ramsey sets
for $F$ might grow at a more moderate rate and be easier to construct these sets
by an inductive argument.

I will finish by mentioning an
intriguing problem concerning
which unbalanced sets are required to witness the non amenability of
all non amenable groups.

\begin{problem}
Is there a finite list $\Bscr$ of unbalanced families such that any non amenable group contains an realization of
an isomorphic copy of an element of $\Bscr$?
\end{problem}

Here two unbalanced families are isomorphic if one is the set-wise image of the other under
a bijection of the underlying sets.


\def\Dbar{\leavevmode\lower.6ex\hbox to 0pt{\hskip-.23ex \accent"16\hss}D}

\end{document}